\documentclass[11pt]{amsart}

\usepackage[english]{babel}
\usepackage[latin1]{inputenc}
\usepackage{amsthm,amsmath,amsfonts, amssymb,bbm}
\usepackage{stmaryrd}

\def \RR{\mathbb R}
\def \EE{\mathbb E}
\def \PP{\mathbb P}
\def \NN{\mathbb N}

\def \SS{\mathbb S}

\def \I{\mathcal I}

\def \N{\mathcal N}

\def \G{\mathcal G}

\def \sclr#1#2{\langle #1,#2\rangle}
\def \open#1{#1^{o}}
\def \close#1{\overline{#1}}
\def \inter#1#2{\llbracket #1,#2\rrbracket}

\renewcommand{\geq}{\geqslant}
\renewcommand{\leq}{\leqslant}

\theoremstyle{plain}
\newtheorem{theorem}{Theorem}
\newtheorem{proposition}[theorem]{Proposition}
\newtheorem{corollary}[theorem]{Corollary}
\newtheorem{lemma}[theorem]{Lemma}

\theoremstyle{definition}
\newtheorem{definition}[theorem]{Definition}

\newtheorem{example}[theorem]{Example}

\begin{document}

\author[R.~Garbit]{Rodolphe Garbit}
\address{Universit\'e d'Angers\\D\'epartement de Math\'ematiques\\ LAREMA\\ UMR CNRS 6093\\ 2 Boulevard Lavoisier\\49045 Angers Cedex 1\\ France}
\email{rodolphe.garbit@univ-angers.fr}

\title[]{On the exit time from an orthant for badly oriented random walks}

\thanks{}

\date{\today}

\begin{abstract} 
It was recently proved that the exponential decreasing rate of the probability that a random walk stays in a $d$-dimensional orthant is given by the minimum on this orthant of the Laplace transform of the random walk increments, provided that this minimum exists. In other cases, the random walk is ``badly oriented'' and the exponential rate may depend on the starting point $x$. We show here that this rate is nevertheless asymptotically equal to the infimum of the Laplace transform, as some selected coordinates of $x$ tend to infinity.
\end{abstract}

\maketitle

\section{Introduction}

\subsection{Context}
This work is a continuation of the paper~\cite{GaRa13} in which the authors studied the exponential decreasing rate of the probability that a random walk (with some exponential moments) stays in a $d$-dimensional convex cone, and found that this rate is equal to the minimum on the dual cone of the Laplace transform of the random walk increments, provided that this minimum exists. In the present work, we shall restrict our attention to the case where the cone is a $d$-dimensional orthant, and extend the preceding result so as to cover also the remaining cases where the random walk is ``badly oriented'' -- a terminology that will be explained later -- with respect to the orthant.

In order to be more specific, let us introduce some notations.
For any fixed probability distribution $\mu$ on $\RR^d$, let $\PP^x_{\mu}$ denote the probability measure on $\RR^{\infty}$ under which the canonical process
$(S_0,S_1,\ldots,S_n,\ldots)$ is a random walk started at $x$ (meaning that $S_0=x$ a.s.) whose independent increments $S_{n+1}-S_n$ have distribution~$\mu$.

Let $K\subset\RR^d$ be some convex cone with non-empty interior and let
$$\tau_K=\inf\{n\geq 1 : S_n\notin K\}$$
denote the exit time of the random walk from $K$.

For random walks with no drift, the precise asymptotics
$$\PP^x_\mu(\tau_K>n)=c \rho^n n^{-\alpha}(1+o(1)),\quad n\to\infty$$
was derived by Denisov and Wachtel in~\cite{DeWa11} from the corresponding tail distribution for Brownian motion by using a strong approximation theorem. In that case $\rho=1$.
They also obtained a local limit theorem from which
Duraj could derive in~\cite{Dur13} the presice asymptotics for random walks with ``negative'' drift, that is when the global minimum on $\RR^d$ of the Laplace transform
$$L_{\mu}(z)=\int_{\RR^d}e^{\sclr{z}{y}}\mu(dy)$$
is reached at an interior point of the dual cone
$$K^*=\{x\in \RR^d : \sclr{x}{y}\geq 0, \forall y\in K\}.$$
In that case, he found that $\rho=\min_{\RR^d}L_{\mu}$. The problem of determining the exponential rate only, but disregarding the position of the global minimum on $\RR^d$, was solved with great generality in \cite{GaRa13}. In that paper, we found that the right place to look at is the position of the {\em global minimum on the dual cone $K^*$} of the Laplace transform.
Indeed, the main result in \cite{GaRa13} is that the exponential decreasing rate
$$\rho_x=\liminf_{n\to\infty}\PP^x_{\mu}(\tau_K>n)^{1/n}$$
is given by the identity
\begin{equation}
\label{old_thm}
\rho_x=\min_{K^*}L_{\mu},
\end{equation}
for all $x$ far enough from the boundary of $K$, provided that this minimum exists. Note that in this case, there is essentially no dependence in the starting point $x$. 

The goal of the present work is to study the case where this minimum does not exist. For technical reasons (that should become clearer when reading the rest of the paper), we shall restrict our attention to the case where $K=Q$ is the positive orthant
$$Q=\{x\in\RR^d : x_i\geq 0, i=1\ldots d\},$$
where $x_i$ denotes the $i$th coordinate of $x$ with respect to the standard basis $(e_1,e_2,\ldots,e_d)$. Note that $Q^*=Q$. In addition, in order to simplify the exposition, we will assume that the probability distribution $\mu$ has all exponential moments, that is, $L_{\mu}(z)$ is finite for all $z\in\RR^d$.

For a truly $d$-dimensional distribution $\mu$, i.e.~a distribution whose support is not included in any linear hyperplane, the condition that $L_{\mu}$ reaches a global minimum on $Q$ is equivalent to the following geometric condition (see \cite{GaRa13} for a proof):
\begin{itemize}
\item[(H)] The support of $\mu$ is not included in any half-space $u^-=\{x\in\RR^d : \sclr{x}{u}\leq 0\}$ with $u\in Q\setminus\{0\}$.
\end{itemize}
Random walks with a distribution $\mu$ that does not fulfill condition (H) are called {\em badly oriented}. In this case, the exponential rate $\rho_x$ may depend on the starting point $x$. 

\begin{example}
Consider the $2$-dimensional lattice distribution $\mu$ defined by
$$\mu(1,-1)=\mu(-1,1)=q, \quad \mu(-1,-1)=p,\quad p+2q=1,\quad p,q>0.$$
The corresponding random walk is badly oriented since the support of $\mu$ is included in $(1,1)^-$.
Its Laplace transform satisfies the relations
$$L_{\mu}(i,j)=2q\cosh(i-j)+pe^{-(i+j)}>2q,$$
and
$$\lim_{i\to +\infty} L_{\mu}(i,i)=2q.$$
Therefore, the infimum $2q$ of $L_{\mu}$  on $Q$ is not a minimum.
It is proved in \cite{GaRa13} that, for $x=(i,j)\in \NN_+^2$ with $i+j=2N$, 
$$\rho_x=2q\cos\left(\frac{\pi}{2N+2}\right).$$
Thus $\rho_x$ depends on $x$, but we nevertheless observe that
$$\lim_{\Vert x\Vert\to\infty}\rho_x=2q=\inf_{Q}L_{\mu}.$$
\end{example}

The aim of this paper is to explain this phenomenon by giving a ``universal'' result that applies to both well and badly oriented random walks. More precisely, we shall prove that the equality~\eqref{old_thm} is in fact a particular case of the following equality
$$\lim\rho_x=\inf_Q L_{\mu},$$
where the interpretation of the limit symbol depends on a linear subspace associated with $\mu$ that we call the {\em reduced support} of $\mu$.

The central idea in~\cite{GaRa13} was to perform a standard Cram\'er transformation (i.e.~an exponential change of measure) based on the point $x_0$ where $L_{\mu}$ reaches its minimum on $Q$.
The main novelty here is that we provide a way to achieve this transformation when $x_0$ is ``at infinity''.

\subsection{Reduction}
The assumption (H) was used in \cite{GaRa13} to ensure the existence of a global minimum of $L_{\mu}$ on $Q$.
The implication follows easily from the following property of the Laplace transform (see \cite[Lemma 3]{GaRa13}):
$$
\lim_{t\to +\infty}L_{\mu}(z+tu)=
\begin{cases}
+\infty &\mbox{ if } \mu(u^-)<1\\
\int_{u^\perp} e^{\sclr{z}{y}}\mu(dy) &\mbox{ if } \mu(u^-)=1,
\end{cases}
$$
where $u^\perp$ denotes the hyperplane orthogonal to $u$.
Indeed, for a distribution satisfying assumption (H), the Laplace transform tends to infinity in each direction $u\in Q\cap\SS^{d-1}$. Since this function is convex, this implies by a compactness argument that $L_{\mu}$ is coercive in $Q$, and the existence of a global minimum follows.

Suppose now that the random walk is badly oriented, i.e.~that there exists a direction $u_1\in Q\cap \SS^{d-1}$ such that $\mu(u_1^-)=1$. Then, the ``best'' way to stay in $Q$ is certainly not to perform a step outside of $u_1^{\perp}$. Thus, it seems natural to compare our original random walk with a conditioned version of it, namely the random walk with distribution $\nu(*)=\mu(*\vert u_1^{\perp})$. Since this new random walk may still be badly oriented, we shall construct by induction a {\em reduced support} $V$ in the following way:

\begin{definition}
\label{def:admissible_maximal}
An $r$-tuple $(u_1,u_2,\ldots, u_r)\in Q^r$ is {\em admissible} if
\begin{itemize}
\item[(i)] the vectors $u_1, u_2, \ldots, u_r$ are linearly independent, and
\item[(ii)] they satisfy the following relations:
$$
\begin{cases}
\mu(u_1^-)=1 &\\
\mu(u_2^-\cap u_1^\perp)=\mu(u_1^\perp) &\\
\vdots &\\
\mu(u_r^{-}\cap (u_{r-1}^\perp\cap\cdots\cap u_1^\perp))=\mu(u_{r-1}^\perp\cap\cdots\cap u_1^\perp).
\end{cases}
$$
\end{itemize}
An admissible $r$-tuple  $(u_1,u_2,\ldots, u_r)$ is {\em maximal} if there is no $u\in Q$ such that $(u_1,u_2,\ldots, u_r, u)$ be admissible.
\end{definition}

If there exists an admissible tuple (i.e. if the random walk is badly oriented), then there exist maximal tuples, and the linear subspace
$[u_1,u_2,\ldots,u_r]$ generated by a maximal tuple does not depend on the specific maximal tuple chosen (Lemma \ref{lem:reduced_support_uniqueness}).

\begin{definition}
\label{def:reduced_support}
The orthogonal complement
$$
V=[u_1,u_2,\ldots,u_r]^\perp
$$
of the subspace generated by any maximal tuple is called the {\em reduced support} of $\mu$.
By convention, if there is no admissible tuple, we set $V=\RR^d$.
\end{definition}

With the reduced support $V$, we associate the set $I$ of indices $i\in\inter{1}{d}$ such that $e_i\in V$, and the cone
$$V^+=\{x\in V : \sclr{x}{e_i}\geq 0, \forall i\in I\}.$$
Then (Lemma~\ref{lem:minimum_reduction}) the infimum on $Q$ of the Laplace transform of $\mu$ is given by the relation
\begin{equation}
\label{eq:inf_decomposition}
\inf_Q L_{\mu}=\inf_{v\in V^+}\int_{V}e^{\sclr{v}{y}}\mu(dy).
\end{equation}
If $\inf_{Q}L_{\mu}=0$, then
$$
\rho_x=\liminf_{n\to\infty}\PP_{\mu}^x(\tau_Q>n)^{1/n}=0
$$
for all $x\in Q$, since $\inf_{Q}L_{\mu}$ is a universal upper bound for $\rho_x$ (see equation~\eqref{eq:universal_upper_bound}).
Thus, from now on, we will exclude this degenerate case by assuming
$$
\inf_{Q}L_{\mu}>0.
$$
Therefore $\mu(V)>0$ and equality~\eqref{eq:inf_decomposition} can be rewritten as
\begin{equation}
\label{eq:inf_decomposition_2}
\inf_Q L_{\mu}=\mu(V)\inf_{V^+} L_{\mu\vert V},
\end{equation}
where $\mu\vert V$ denotes the conditional distribution $\mu\vert V(*)=\mu(*\vert V)$. Now, the ``maximality'' of $V$ (in the sense of Definition~\ref{def:admissible_maximal}) ensures
that the conditioned random walk with distribution $\mu\vert V$ is well oriented with respect to $V^+$, i.e.~that the infimum on $V^+$ is in fact a minimum
(Lemma \ref{lem:existence_of_global_minimum}), so that a Cram\'er transformation can be applied to this conditioned random walk. We interpret this as a Cram\'er transformation at infinity.

\subsection{Main result}
In what follows, $\mu$ is any probability distribution on $\RR^d$ with all exponential moments.
To avoid trivialities, we assume
$$\inf_{Q}L_{\mu}>0.$$
We denote by $V$ the reduced support of $\mu$ (see Definitions \ref{def:admissible_maximal} and \ref{def:reduced_support}), and define
$$I=\{i\in\inter{1}{d}: e_i\in V\},\quad I^\perp=\{i\in\inter{1}{d}: e_i\in V^\perp\},$$
and, for $x$ in $Q$,
$$d(x)=\min_{i\notin I\cup I^\perp}{x_i},$$
with the convention that $d(x)=\infty$ when $I\cup I^\perp=\inter{1}{d}$. 
Finally, we set
$$Q_{\delta}=Q+\delta(1,1,\ldots,1)=\{x\in \RR^d: x_i\geq \delta, \forall i\in\inter{1}{d}\}.$$
We are now in position to state our main result.

\begin{theorem}
\label{thm:main_theorem}
There exists $\delta\geq 0$ such that
$$\lim_{d(x)\to \infty \atop x\in Q_{\delta}}\rho_x=\inf_{Q}L_{\mu}.$$
\end{theorem}
Let us illustrate this theorem with some basic examples.

\begin{example}
If $\mu$ satisfies assumption (H), then $V=\RR^d$, $I=\inter{1}{d}$, $V^+=Q$, $I^\perp=\emptyset$, and $d(x)=\infty$. Thus we recover the non-asymptotic theorem of~\cite{GaRa13}: There exists $\delta\geq 0$ such that
$$\rho_x=\inf_{Q}L_{\mu},$$
for all $x\in Q_{\delta}$.
\end{example}

\begin{example}
Consider the $2$-dimensional lattice distribution $\mu$ defined by
$$\mu(-1,0)=\alpha,\quad\mu(0,1)=\beta,\quad\mu(0,-1)=\gamma,\quad \alpha+\beta+\gamma=1,\quad \alpha,\beta,\gamma>0.$$
The associated random walk is badly oriented since $\mu(e_1^-)=1$. Now, $(e_1)$ is maximal since $\beta,\gamma>0$, therefore
$V=e_1^\perp$, $I=\{2\}$, $V^+=\{(0,j): j\geq 0\}$,  $I^\perp=\{1\}$ and $d(x)=\infty$, meaning that 
$$\rho_x=\inf_{Q}L_{\mu},$$
for all $x\in Q_{\delta}$, for some constant $\delta\geq 0$. Let us compute the value of $\rho_x$. The
Laplace transform of $\mu$ is given by
$$L_{\mu}(i,j)=\alpha e^{-i}+\beta e^j+\gamma e^{-j}.$$
Minimizing first on $i\geq 0$ leads to the relation
$$\inf_{Q}L_{\mu}=\inf_{j\geq 0}(\beta e^j+\gamma e^{-j})=(\beta+\gamma)\inf_{j\geq 0}(\beta' e^j+\gamma' e^{-j}),$$
where $\beta'=\beta/(\beta+\gamma)$ and $\gamma'=\gamma/(\beta+\gamma)$ sum up to $1$. 
Notice that the above formula corresponds exactly to equality \eqref{eq:inf_decomposition_2} with $V=e_1^\perp$.
An easy computation shows that
$$\inf_{j\geq 0}(\beta' e^j+\gamma' e^{-j})=
\begin{cases}
2\sqrt{\beta'\gamma'} &\mbox{ if } \gamma'\geq \beta'\\
1 &\mbox{ else.}
\end{cases}
$$
Hence 
$$
\rho_x=\inf_Q L_{\mu}=
\begin{cases}
2\sqrt{\beta\gamma} &\mbox{ if } \gamma\geq \beta\\
\beta+\gamma &\mbox{ else.}
\end{cases}
$$
\end{example}

\begin{example}
Consider the $2$-dimensional lattice distribution $\mu$ defined by
$$\mu(-1,-1)=\alpha,\quad\mu(-1,1)=\beta,\quad\mu(1,-1)=\gamma,\quad \alpha+\beta+\gamma=1,\quad \alpha,\beta,\gamma>0.$$
The associated random walk is badly oriented since the support of $\mu$ is included in $(1,1)^-$. Here,
 $V=(1,1)^\perp$, $I=I^\perp=\emptyset$, $V^+=V$ and $d(x)=\min\{x_1,x_2\}$.
Therefore, we obtain
$$\lim_{x_1,x_2\to\infty}\rho_x=\inf_{Q}L_{\mu}=2\sqrt{\beta\gamma},$$
as the reader may check.
\end{example}

The rest of the paper is organized as follows: 
In Section~\ref{sec:proof_main_thm}, we present the proof of Theorem~\ref{thm:main_theorem}. The construction of the reduced support enables us to perform a Cram\'er's transformation at infinity, and then compare our initial exponential decreasing rate with that of a conditioned random walk whose distribution is more ``favorable''. The exponential rate of this conditioned random walk is
then analysed in Section~\ref{sec:favorable_case}. The Appendix at the end of the paper provides some material on polyhedra that is needed in a technical lemma and for which no reference were found.

\section{Proof of Theorem~\ref{thm:main_theorem}}
\label{sec:proof_main_thm}

The first subsection recalls the strategy used in \cite{GaRa13} in order to obtain the exponential decreasing rate in the case where the random walk is well oriented.
Though not necessary, its reading is recommended in order to become acquainted with the basic ideas of the proof of Theorem~\ref{thm:main_theorem}. We will then present the construction and the main properties of the reduced support, and finally use it to perform our Cram\'er's transformation at infinity in order to conclude the proof of Theorem~\ref{thm:main_theorem}.

\subsection{Sketch of proof for well oriented random walks}

The reader may know, or at least should let his intuition convince himself that $\rho_x=1$ when the drift vector $m=\EE(\mu)$ belongs to the orthant $Q$. In other cases, the basic idea is to carry out an exponential change of measure in order to transform the drift vector $m$ into a new one, $m_0$, that belongs to $Q$. To do this, fix any $x_0\in \RR^d$ and define
$$\mu_0(dy)=\frac{e^{\sclr{x_0}{y}}}{L_{\mu}(x_0)}\mu(dy).$$
Clearly, $\mu_0$ is a probability distribution with Laplace transform
$$L_{\mu_0}(z)=\frac{L_{\mu}(z+x_0)}{L_{\mu}(x_0)}.$$
Furthermore, it is well-known that the expectation of a probability distribution is equal to the gradient of its Laplace transform evaluated at $0$.
Thus
$$m_0=\EE(\mu_0)=\nabla L_{\mu}(x_0)/L_{\mu}(x_0).$$
If $L_{\mu}$ reaches a global minimum on $Q$ at $x_0$, then its partial derivatives are non-negative at $x_0$ and, therefore,
the drift $m_0$ belongs to $Q$. So, the distribution $\mu_0$ is ``nice'' in the sense that the value of the exponential rate for a random walk with that distribution is known.
The link between the distribution of our canonical random walk under $\PP_{\mu}^x$ and $\PP_{\mu_0}^x$ is expressed via Cram\'er's formula (see \cite[Lemma 6]{GaRa13} for example) which leads in our case to the relation
\begin{equation}
\label{eq:cramer_0}
\PP_{\mu}^x(\tau_Q>n)=L_{\mu}(x_0)^ne^{\sclr{x_0}{x}}\EE_{\mu_0}^x\left(e^{-\sclr{x_0}{S_n}},\tau_Q>n\right).
\end{equation}
From this, we already notice that for all $x_0\in Q$ (being a minimum point or not),
$$\limsup_{n\to\infty}\PP_{\mu}^x(\tau_Q>n)^{1/n}\leq L_{\mu}(x_0),$$
since $\sclr{x_0}{S_n}\geq 0$ as soon as $S_n\in Q$. Thus, the upper bound
\begin{equation}
\label{eq:universal_upper_bound}
\limsup_{n\to\infty}\PP_{\mu}^x(\tau_Q>n)^{1/n}\leq \inf_Q L_{\mu}
\end{equation}
holds for any probability distribution $\mu$ (with all exponential moments).

In order to obtain the corresponding lower bound
in the case where $L_{\mu}$ reaches a global minimum on $Q$ at $x_0$, we use formula \eqref{eq:cramer_0} with our nice distribution $\mu_0$. What is not so nice here is the
$\exp(-\sclr{x_0}{S_n})$ term, which could become very small since $(S_n)$ has a drift $m_0\in Q$. But, in fact, the growth of $\sclr{x_0}{S_n}$ can be controled thanks to the following observation: Let $K$ be the set of indices $i$ such that $x_0^{(i)}>0$. Since $x_0$ belongs to $Q$, the other coordinates are equal to zero, and
$$
\sclr{x_0}{S_n}=\sum_{i\in K}x_0^{(i)}S_n^{(i)}.
$$
Furthermore, if $x_0^{(i)}>0$ for some index $i$, then $0$ is a local minimum of the partial function $t\in [-x_0^{(i)}, +\infty)\mapsto L_{\mu}(x_0+te_i)$ and therefore $m_0^{(i)}=0$, because it is proportional to $\partial_{x_i} L_{\mu}(x_0)=0$. So, the coordinates of the random walk that we need to control have zero mean. Since $\sqrt{n}$ is a natural bound for the norm of a square integrable centered random walk, one can expect that adding the constraint 
\begin{equation}
\label{eq:constraint}
\max_{i\in K}\vert S^{(i)}_n\vert\leq\sqrt{n}
\end{equation}
will not decrease too much the probability on the right-hand side of equation \eqref{eq:cramer_0}.
As proved in~\cite{GaRa13} this happens to be true: For any distribution $\mu$ satisfying assumption (H),
$$
\liminf_{n\to\infty}\PP_{\mu_0}^x\left(\max_{i\in K}\vert S^{(i)}_n\vert\leq\sqrt{n},\tau_Q>n\right)^{1/n}=1.
$$
Under the constraint \eqref{eq:constraint}, the exponential term  $\exp(-\sclr{x_0}{S_n})$ in~\eqref{eq:cramer_0} is bounded from below by $\exp(-\Vert x_0\Vert\sqrt{n})$, a term that disappears in the $n$th root limit.
Thus, we obtain the lower bound
$$
\rho_x \geq L_{\mu}(x_0)\liminf_{n\to\infty}\PP_{\mu_0}^x\left(\max_{i\in K}\vert S^{(i)}_n\vert\leq\sqrt{n},\tau_Q>n\right)^{1/n}=L_{\mu}(x_0),
$$
and this concludes the analysis.

\subsection{Construction and properties of the reduced support}

In this section, we prove the uniqueness of the reduced support of $\mu$ (Definitions \ref{def:admissible_maximal} and \ref{def:reduced_support}) and establish some of its properties.  When $\mu$ satisfy assumption (H), the reduced support $V$ is equal to $\RR^d$ -- by definition -- and all properties listed below are trivial, as the reader may check. Therefore, in what follows, we assume the existence of a direction $u_1\in Q\cap\SS^{d-1}$ such that $\mu(u_1^-)=1$.
Thus, there exists at least an admissible tuple, namely $(u_1)$, and the existence of a maximal tuple follows immediately.

First, we prove the uniqueness\footnote{This property will not be used, but it is a very natural question since the conclusion of Theorem~\ref{thm:main_theorem} depends on this subspace.} of the linear space generated by maximal admissible tuples.

\begin{lemma}
\label{lem:reduced_support_uniqueness}
Any two maximal admissible tuples generate the same linear space.
\end{lemma}
\begin{proof}
Let $(u_1,u_2,\ldots,u_r)$ and $(v_1,v_2,\ldots,v_s)$ be two maximal admissible tuples and let
$$A=[u_1,u_2,\ldots,u_r]\quad\mbox{ and }\quad B=[v_1,v_2,\ldots,v_s]$$
denote the linear spaces generated by this two tuples, respectively.
Since $(u_1,u_2,\ldots,u_r)$ is maximal, $(u_1,u_2,\ldots,u_r, v_1)$ is not admissible. On the other hand, $\mu(v_1^-)=1$ so that
$$\mu(v_1^-\cap A^\perp)=\mu(A^\perp),$$
and the tuple $(u_1,u_2,\ldots,u_r, v_1)$ satisfies condition (ii) of Definition~\ref{def:admissible_maximal}. Thus, the reason why this tuple fails to be admissible is that condition (i) is not satisfied, i.e.~$v_1$ belongs to $A$. Now, suppose that $v_1,v_2,\ldots, v_{k-1}$ all belong to $A$ for some $k\in\inter{2}{s}$, and denote
$B_k=[v_1,v_2,\ldots, v_{k-1}]$ the linear space they generate. By hypothesis, we have
$$A^\perp\subset B_k^\perp\quad\mbox{ and }\quad \mu(v_k^-\cap B_k^{\perp})=\mu(B_k^\perp).$$
This implies $\mu(v_k^-\cap A^{\perp})=\mu(A^\perp)$ and therefore $v_k$ must belong to $A$, otherwise $(u_1,u_2,\ldots,u_r, v_k)$ would be admissible.
By induction, we obtain the inclusion $B\subset A$, and the equality $A=B$ follows by interchanging the role of $A$ and $B$.
\end{proof}

Thanks to Lemma~\ref{lem:reduced_support_uniqueness} we can define the {\em reduced support} $V$ of $\mu$ as the orthogonal complement of {\em the}
linear space generated by any maximal admissible tuple (Definition~\ref{def:reduced_support}). In what follows, we fix a maximal admissible tuple
$(u_1,u_2,\ldots, u_r)$ and set 
\begin{equation*}
V=[u_1,u_2,\ldots,u_r]^\perp\quad\mbox{ and }\quad V^+=V\cap\left(\bigcup_{h\in V^\perp}Q-h\right).
\end{equation*}
The reduced support $V$ and the cone $V^+$ play a fundamental role in our analysis. First of all, they provide a useful expression for the infimum of the Laplace transform.

\begin{lemma}
\label{lem:minimum_reduction}
The following equality holds:
\begin{equation*}
\inf_Q L_{\mu}=\inf_{v\in V^+}\int_{V}e^{\sclr{v}{y}}\mu(dy).
\end{equation*}
\end{lemma}
\begin{proof} 
It follows from the orthogonal decomposition $\RR^d=V\oplus V^{\perp}$ that
\begin{align*}
\inf \{L_{\mu}(x) : x\in Q\} & =\inf\{L_{\mu}(v+h):(v,h)\in V\times V^\perp, v+h\in Q\} \\
                             & =\inf_{v\in V^+}\inf_{h\in V^\perp\cap (Q-v)}L_{\mu}(v+h).
\end{align*}
Therefore, we have to prove that for all $v\in V^+$
\begin{equation}
\label{inf_is_conditional_laplace}
\inf_{h\in V^\perp\cap (Q-v)}L_{\mu}(v+h)=\int_{V}e^{\sclr{v}{y}}\mu(dy).
\end{equation}
Let $v\in V^+$. First, for all $h\in V^{\perp}$, we notice that
\begin{equation*}
L_{\mu}(v+h)\geq \int_V e^{\sclr{v+h}{y}}\mu(dy)=\int_V e^{\sclr{v}{y}}\mu(dy),
\end{equation*}
since $\sclr{h}{y}=0$ for all $y\in V$. Now, pick $h\in V^\perp\cap (Q-v)$ (such a $h$ exists since $v\in V^+$), and let $\lambda_1,\lambda_2,\ldots,\lambda_r$ denote its coordinates with respect to the basis $(u_1,u_2,\ldots,u_r)$ of $V^{\perp}$. It is clear that $h=\sum_{k=1}^r\lambda_ku_k$ will still belong to $V^\perp\cap (Q-v)$ if one increases the value of any $\lambda_i$. Hence, the equality~\eqref{inf_is_conditional_laplace} will follow from
\begin{equation}
\label{reaching_conditional_laplace}
\lim_{\lambda_r\to\infty}\ldots\lim_{\lambda_2\to\infty}\lim_{\lambda_1\to\infty}L_{\mu}\left(v+\sum_{k=1}^r\lambda_r u_r\right)=\int_V e^{\sclr{v}{y}}\mu(dy).
\end{equation}
Since $(u_1,u_2,\ldots,u_r)$ is admissible, $\mu(u_1^-)=1$ and consequently
\begin{equation*}
L_{\mu}\left(v+\sum_{k=1}^r\lambda_r u_r\right)
=\int_{u_1^\perp}e^{\sclr{v}{y}}e^{\sum_{k=2}^r\lambda_k\sclr{u_k}{y}}\mu(dy)+\int_{\{\sclr{u_1}{y}<0\}}e^{\sclr{v}{y}}e^{\sum_{k=1}^r\lambda_k\sclr{u_k}{y}}\mu(dy).
\end{equation*}
By the dominated convergence theorem, the second integral on the right-hand side of the above equation goes to zero as $\lambda_1$ goes to infinity. Hence
\begin{equation*}
\lim_{\lambda_1\to\infty}L_{\mu}\left(v+\sum_{k=1}^r\lambda_r u_r\right)=\int_{u_1^\perp}e^{\sclr{v}{y}}e^{\sum_{k=2}^r\lambda_k\sclr{u_k}{y}}\mu(dy).
\end{equation*}
Now, by hypothesis, we have $\mu(u_2^-\cap u_1^\perp)=\mu(u_1^\perp)$, so that the same argument as above leads to
\begin{equation*}
\lim_{\lambda_2\to\infty}\lim_{\lambda_1\to\infty}L_{\mu}\left(v+\sum_{k=1}^r\lambda_r u_r\right)=\int_{u_1^\perp\cap u_2^\perp}e^{\sclr{v}{y}}e^{\sum_{k=3}^r\lambda_k\sclr{u_k}{y}}\mu(dy).
\end{equation*}
The equality in~\eqref{reaching_conditional_laplace} is obtained by using repeatedly the same argument.
\end{proof}

\begin{lemma}
\label{lem:existence_of_global_minimum}
Assume $\inf_Q L_{\mu}>0$. Then:
\begin{enumerate}
\item $\mu(V)>0$.
\item For all $v\in V^+\cap\SS^{d-1}$, $\mu(v^-\vert V)<1$.
\item The Laplace transform $L_{\mu\vert V}$ has a global minimum on~$V^+$.
\item $\inf_Q L_{\mu}=\mu(V)\min_{V^+}L_{\mu\vert V}$.
\end{enumerate}
\end{lemma}
\begin{proof} 
The formula in Lemma \ref{lem:minimum_reduction} shows that $\inf_Q L_{\mu}=0$ as soon as $\mu(V)=0$. The first item follows by contraposition.

Let $v\in V^+\cap\SS^{d-1}$. By definition of $V^+$ there exists $h\in V^\perp$ such that $u=v+h\in Q$.
For any $y\in V$, we have $\sclr{v}{y}=\sclr{u}{y}$ and consequently
$$
v^-\cap V=u^-\cap V.
$$
So, it suffices to show that $\mu(u^-\vert V)<1$ and the second assertion of the lemma will follow.
By definition, $u=v+h$ (with $v\in V$, $v\not=0$ and $h\in V^{\perp}$) is linearly independent of $V^\perp=[u_1,u_2,\ldots,u_r]$ and
belongs to $Q$. But by maximality, the tuple $(u_1,u_2,\ldots,u_r,u)$ is not admissible.
Thus, we must have $\mu(u^-\cap V)<\mu(V)$. This proves the second assertion of the lemma.

The third assertion follows from the second one since $V^+$ is a closed cone (see \cite[Lemma 4]{GaRa13} -- note that the hypothesis (H1) is not used for the part of the lemma that we need here).

Finally, the last item is just a reformulation of the formula in Lemma \ref{lem:minimum_reduction}.
\end{proof}

We shall now give a very simple description of the cone $V^+$ associated with the reduced support $V$.
To this end, we define
$$I=\{i\in\inter{1}{d} : e_i\in V\}.$$
\begin{lemma}
\label{lem:description_of_V_plus}
The cone $V^+$ has the following expression: 
$$V^+=\{x\in V : \sclr{x}{e_i}\geq 0, \forall i\in I\}.$$ 
\end{lemma}
\begin{proof} We first note that $x\in \bigcup_{h\in V^\perp}Q-h$ if and only if there exists $h\in V^{\perp}$ such that $\sclr{x+h}{e_i}\geq 0$ for all $i$. 
But, for $i\in I$, $\sclr{x+h}{e_i}=\sclr{x}{e_i}$ since $e_i\in V$ and $h\in V^{\perp}$. Hence, the condition splits into:
\begin{enumerate}
\item $\sclr{x}{e_i}\geq 0$ for all $i\in I$, and
\item\label{cond_v_plus} there exists $h\in V^\perp$ such that $\sclr{x+h}{e_i}\geq 0$ for all $i\notin I$.
\end{enumerate}
Therefore, it remains to prove that the last condition holds for all $x$. To this end,
recall that $V^\perp=[u_1,u_2,\ldots, u_r]$ with $u_k\in Q$, and set $u_0=\sum_{k=1}^ru_k$.
Since $\sclr{u_0}{e_i}=\sum_{k=1}^r\sclr{u_k}{e_i}$ and $\sclr{u_k}{e_i}\geq 0$ for all $k$, we see that
$\sclr{u_0}{e_i}\geq 0$
and equality occurs if and only if $\sclr{u_k}{e_i}=0$ for all $k$, which means exactly that $e_i\in [u_1,u_2,\ldots,u_r]^{\perp}=V$. Thus, by definition of the set $I$, we have
$\sclr{u_0}{e_i}>0$ for all $i\notin I$. For any fixed $x$, this property allows to find $\lambda>0$ such that $\sclr{x}{e_i}+\lambda\sclr{u_0}{e_i}\geq 0$ for all $i\notin I$,  so that condition~\eqref{cond_v_plus} holds with $h=\lambda u_0$. This proves the lemma.
\end{proof}

\subsection{Comparison with the conditioned random walk with increments restricted to the reduced support}

As mentioned earlier (see \eqref{eq:universal_upper_bound}), the infimum of $L_{\mu}$ on $Q$ is always an upper bound for the exponential decreasing rate
$$\rho_x=\liminf_{n\to\infty}\PP^x_{\mu}(\tau_Q>n)^{1/n}.$$
Therefore, our task is to show that it is also a lower bound for $\rho_x$ (at least as $x\to\infty$ in the sense of Theorem~\ref{thm:main_theorem}).

Let $V$ be the reduced support of $\mu$. From now on, we assume that $\inf_Q L_{\mu}>0$, so that $\mu(V)>0$ (Lemma~\ref{lem:existence_of_global_minimum}).
Remember that we have introduced the reduced support with the idea that the best way to stay in $Q$ was to never perform any step outside of $V$.
Hence, denoting $\xi_1,\xi_2,\ldots, \xi_n$ the increments of the random walk, it is natural to use the lower bound
\begin{equation}
\label{eq:probability_reduction}
\PP^x_{\mu}(\tau_Q>n)^{1/n}\geq \PP^x_{\mu}(\tau_Q>n, \xi_1,\xi_2,\ldots, \xi_n\in V)^{1/n}=\mu(V)\PP^x_{\mu\vert V}(\tau_Q>n)^{1/n}.
\end{equation}
A look at the last formula of Lemma \ref{lem:existence_of_global_minimum},
\begin{equation}
\label{eq:infimum_reduction}
\inf_Q L_{\mu}=\mu(V)\min_{V^+}L_{\mu\vert V},
\end{equation}
then explains our strategy: Theorem~\ref{thm:main_theorem} will follow from a comparison between 
$$\liminf_{n\to\infty}\PP^x_{\mu\vert V}(\tau_Q>n)^{1/n}\quad\mbox{ and }\quad \min_{V^+}L_{\mu\vert V}.$$

To simplify notations, set $\nu=\mu\vert V$.
For all $x$ in $Q$, write $x=v+w$ the orthogonal decomposition with respect to $V$ and $V^\perp$.
Then
$$\PP^x_{\nu}(\tau_Q>n)=\PP^v_{\nu}(S_1, S_2, \ldots,S_n\in Q-w).$$
Under $\PP^{v}_{\nu}$, the random walk $S_1, S_2,\ldots, S_n$ almost surely belongs to $V$ (since $v\in V$, and $\nu(V)=1$). Thus,
we have to focus our attention on the geometry
of $(Q-w)\cap V$.

To this end, we recall that 
$$I=\{i\in\inter{1}{d}: e_i\in V\}.$$
We also define
$$I^\perp=\{i\in\inter{1}{d}: e_i\in V^\perp\},$$
and, for $x\in Q$,
$$d(x)=\min_{i\notin I\cup I^\perp}x_i.$$
Notice that $\Vert a-x\Vert\leq d(x)$ implies that $a_i\geq 0$ for all $i\notin I\cup I^\perp$.

Let $V_1=[e_i,i\in I]$ and write $V=V_1\oplus V_2$ the orthogonal decomposition of $V$.
Define the positive orthant of $V_1$ as
$$V_1^+=\{y\in V_1: \sclr{y}{e_i}\geq 0, \forall i\in I\},$$ 
and notice that Lemma~\ref{lem:description_of_V_plus} asserts that
$$V^+=V_1^+\oplus V_2.$$
For any $y\in V$, let $y^{(1)}$ and $y^{(2)}$ be the projections of $y$ onto $V_ 1$ and $V_ 2$, respectively.

\begin{lemma} 
\label{lem:set_where_conditioned_walk_should_stay}
For all $x=v+w\in Q$, holds the inclusion:
$$\{y\in V: y^{(1)}\in V_1^+\mbox{ and }\Vert y^{(2)}-v^{(2)}\Vert\leq d(x) \}\subset (Q-w)\cap V.$$
\end{lemma}
\begin{proof}
Let $y\in V$ be such that $y^{(1)}\in V_1^+$ and $\Vert y^{(2)}-v^{(2)}\Vert\leq d(x)$. We have to show that $y+w$ belongs to $Q$.

First of all, for any $i\in I^{\perp}$, we have $e_i\in V^\perp$ so that
$$\sclr{y+w}{e_i}=\sclr{v+w}{e_i}=x_i\geq 0.$$

Similarly, since $\sclr{y+w}{e_i}=\sclr{y^{(1)}}{e_i}$ for all $i\in I$, the condition $y^{(1)}\in V_1^+$  rewrites
$\sclr{y+w}{e_i}\geq 0$ for all $i\in I$.

It remains to check that the conclusion also holds when $i\notin I\cup I^\perp$. To this end, we notice that
$$\Vert (v^{(1)}+y^{(2)}+w)-x\Vert=\Vert y^{(2)}-v^{(2)}\Vert\leq d(x).$$
Therefore, for all $i\notin I\cup I^\perp$, 
$$\sclr{v^{(1)}+y^{(2)}+w}{e_i}\geq 0.$$
But for those indices $i$, we have $e_i\perp V_1$ and consequently
$$\sclr{y+w}{e_i}=\sclr{v^{(1)}+y^{(2)}+w}{e_i}\geq 0.$$
This concludes the proof of the lemma.
\end{proof}

This lemma provides the convenient lower bound: 
\begin{equation}
\label{eq:lower_bound}
\PP^x_{\nu}(\tau_Q>n)\geq \PP^v_{\nu}(\tau_{V^+}>n,  \max_{k\leq n}\Vert S^{(2)}_k-v^{(2)}\Vert\leq d(x)).
\end{equation}
We now analyse this lower bound with the help of Cram\'er's transformation.
We know by Lemma~\ref{lem:existence_of_global_minimum} that there exists $v_0\in V^+$ such that
$$\lambda:=L_{\nu}(v_0)=\min_{V^+}L_{\nu}>0.$$
Let $\nu_0$ be the probability measure on $V$ defined by
$$\lambda \nu_0(dy)=e^{\sclr{v_0}{y}}\nu(dy).$$
Thanks to Cram\'er's formula (see \cite[Lemma 6]{GaRa13}), the lower bound in equation \eqref{eq:lower_bound} can be written as
\begin{equation}
\label{eq:lower_bound_2}
\lambda^n\EE^{v}_{\nu_0}\left(e^{-\sclr{v_0}{S_n-v}}, \tau_{V^+}>n,  \max_{k\leq n}\Vert S^{(2)}_k-v^{(2)}\Vert\leq d(x)\right).
\end{equation}
Since $v_0\in V^+$, we have $\sclr{v_0}{e_i}\geq 0$ for all $i\in I$. Define 
$$K=\{i\in I: \sclr{v_0}{e_i}>0\}.$$
Then $\sclr{v_0}{e_i}=0$ for all $i\in I\setminus K$, so that
\begin{align*}
\vert \sclr{v_0}{S_n-v}\vert &=\vert \sum_{i\in K}\sclr{v_0}{e_i}\sclr{S^{(1)}_n-v^{(1)}}{e_i}+\sclr{v_0}{S^{(2)}_n-v^{(2)}}\vert\\
&\leq (\sum_{i\in K}\sclr{v_0}{e_i}+\Vert v_0\Vert) d(x),
\end{align*}
as soon as $\vert\sclr{S^{(1)}_n-v^{(1)}}{e_i}\vert\leq d(x)$ for all $i\in K$ and $\Vert S^{(2)}_n-v^{(2)}\Vert\leq d(x)$.
Under this additional constraint, the term $\exp{(-\sclr{v_0}{S_n})}$ inside the expectation in \eqref{eq:lower_bound_2} is bounded from below by some positive constant that will disappear in the $n$th root limit. 

Therefore, using the notation
$$\widetilde{\rho}_x=\liminf_{n\to\infty}\PP^x_{\nu}(\tau_Q>n)^{1/n},$$
we have
\begin{equation}
\label{eq:lower_bound_3}
\widetilde{\rho}_x\geq \lambda
\liminf_{n\to\infty}\PP^v_{\nu_0}\left(\tau_{V^+}>n, \max_{k\leq n} \chi(S_k-v)\leq d(x)\right)^{1/n}
\end{equation}
where 
$$\chi(S_k-v)=\max\{\max_{i\in K}\vert\sclr{S^{(1)}_k-v^{(1)}}{e_i}\vert, \Vert S^{(2)}_k-v^{(2)}\Vert\}.$$

The behavior of this last limit is analysed in Theorem~\ref{prop:centered_case_extended}. Indeed, the space $V$ where the probability distribution $\nu$ lives
has the cartesian product structure $V=V_1\oplus V_2$ and the cone writes $V^+=V^+_1 \oplus V_2$, where $V^+_1$ is the positive orthant of $V_1$. Furthermore, Lemma~\ref{lem:existence_of_global_minimum} asserts that $\nu=\mu\vert V$ satisfies assumption (H') of Theorem~\ref{prop:centered_case_extended} with respect to $V^+$; so does $\nu_0$ since $\nu$ and $\nu_0$ are absolutely continuous with respect to each other. Thus, it remains to take a look at its expectation $m_0$ which is given by
$$
m_0=\nabla L_{\nu_0}(0)=\lambda^{-1}\nabla L_{\nu}(v_0).
$$
Let us write $m_0=m_0^{(1)}+m_0^{(2)}$ with $m_0^{(1)}\in V_1$ and $m_0^{(2)}\in V_2$.
Since $v_0$ is a global minimum point of $L_{\nu}$ on $V^+=V_1^+\oplus V_2$, it is easily seen that:
\begin{itemize}
\item $m_0^{(1)}\in V_1^+$,
\item $\sclr{m_0^{(1)}}{e_i}=0$ for all $i\in K$, and 
\item $m_0^{(2)}=0$.
\end{itemize}
Indeed, for all $i\in I$, we have $v_i=\sclr{v_0}{e_i}\geq 0$, and the half-line
$\{v_0+te_i : t\geq -v_i\}$ is included in $V^+$. Therefore, the function
$$t\in[-v_i,\infty)\mapsto f_i(t)=L_{\nu}(v_0+te_i)$$
reaches its minimum at $t=0$. This implies that $\sclr{m_0^{(1)}}{e_i}=\lambda^{-1}f'_i(0)\geq 0$ with equality if $v_i>0$, i.e. if $i\in K$. 
On the other hand, if we take any $h\in V_2$
then the whole line $\{v_0+th : t\in \RR\}$ is included in $V^+$, hence the function $t\in\RR\mapsto L_{\nu}(v_0+th)$ reaches a local minimum at $t=0$.
By consequence,  its derivative $\sclr{\nabla L_{\nu}(0)}{h}=\lambda\sclr{m_0}{h}$ at $t=0$ is always equal to $0$. Thus $m_0^{(2)}=0$. 

Thanks to those properties of $\nu_0$ and its expectation $m_0$, we can apply Theorem~\ref{prop:centered_case_extended} which ensures the existence of some
$\delta\geq 0$ such that
$$
\lim_{R\to\infty}\liminf_{n\to\infty}\PP^v_{\nu_0}\left(\tau_{V^+}>n, \max_{k\leq n} \chi(S_k-v)\leq R\right)^{1/n}=1,
$$
for all $v\in V^+_\delta=\{v\in V: v_i\geq\delta, \forall i\in I\}$. Notice that the probability under consideration reaches its minimum on $V^+_{\delta}$ when $v$ is the ``corner''
point $v_*=\delta\sum_{i\in I}e_i$ (this follows by inclusion of events). Hence it follows from \eqref{eq:lower_bound_3}
that  for any $x=v+w\in Q$ with $v\in V^+_\delta$,
$$
\widetilde{\rho}_x\geq \lambda
\liminf_{n\to\infty}\PP^{v_*}_{\nu_0}\left(\tau_{V^+}>n, \max_{k\leq n} \chi(S_k-v_*)\leq d(x)\right)^{1/n}.
$$
Since $d(x)$ is now disconnected from the starting point $v_*$ we can let $d(x)\to\infty$, thus proving that
$$
\lim_{d(x)\to\infty\atop x\in Q_\delta}\widetilde{\rho}_x\geq \lambda=\min_{V^+}L_{\mu\vert V}.
$$
Theorem~\ref{thm:main_theorem} then follows from the combination of this inequality with \eqref{eq:probability_reduction} and \eqref{eq:infimum_reduction}.

\section{The favorable case}
\label{sec:favorable_case}

In this section, we consider a square integrable random walk $(S_n)$ in $\RR^d=\RR^p\times\RR^q$ with distribution $\mu$, mean $m$ and variance-covariance matrix $\Gamma$.
For all $x\in\RR^d$, we denote by $x_i$, $i=1\ldots d$, its coordinates in the standard basis, and $x^{(1)}\in\RR^p$ and $x^{(2)}\in\RR^q$ its ``coordinates'' with respect to the cartesian product $\RR^p\times\RR^q$.
Let $Q$ denote the positive orthant of $\RR^p$, i.e.
$$
Q=\{x\in\RR^p: x_i\geq 0, \forall i=1\ldots p\}.
$$
We are interested in the tail distribution of the exit time
$$\tau_K=\inf\{n\geq 1 : S_n\notin K\}$$
of the random walk from the cartesian product
$$K=Q\times\RR^q,$$
in the {\em favorable} case where $m\in Q\times\{0\}^q$, but with the additional constraint that some zero-mean coordinates of the walk stay in a bounded domain. In what follows, we assume the random walk is well oriented (with respect to $K$), i.e.~the probability distribution satisfies the following condition:
\begin{itemize}
\item[(H')] The support of $\mu$ is not included in any half-space $u^-=\{x\in\RR^d : \sclr{x}{u}\leq 0\}$ with $u\in Q\times\{0\}^q\setminus\{0\}$.
\end{itemize}

Let $J$ be the set of indices $j$ such that $m_j=0$.
Denote by $\RR^{(J)}$ the subspace of $J$-coordinates, that is
$$\RR^{(J)}=\{x\in \RR^{d}: x_i=0, \forall i\notin J\},$$
and let
$x^{(J)}$ be the projection of $x\in\RR^d$ on $\RR^{(J)}$. Let also
\begin{equation}
\label{eq:norme_ji}
\Vert x\Vert_J=\Vert x^{(J)}\Vert=\left(\sum_{i\in J}\vert x_i\vert^2\right)^{1/2}
\end{equation}
be the norm of the projection of $x$ on the subspace of $J$-coordinates.

We shall prove in this setting the following result that extends Theorem~13 of~\cite{GaRa13}.

\begin{theorem}
\label{prop:centered_case_extended}
Assume $\mu$ satisfies (H'), and that $m\in Q\times\{0\}^q$. Let $J$ be the set of indices $j$ such that $m_j=0$. There exists $\delta\geq 0$ such that
$$\lim_{R\to\infty}\liminf_{n\to\infty}\PP_{\mu}^x\left(\tau_K>n, \max_{k\leq n}\Vert S_k-x\Vert_J \leq R\right)^{1/n}=1$$
for all $x\in K_{\delta}=Q_{\delta}\times\RR^q$.
\end{theorem}

Theorem~\ref{prop:centered_case_extended} will follow from the two propositions below. Roughly speaking, the first one  will enable us to push the random walk as far as we want from the boundary
of the cone $K$ (with a positive probability):

\begin{proposition}
\label{prop:n_steps_push_forward} Under the hypotheses of Theorem~\ref{prop:centered_case_extended},
there exist $\gamma>0$, $b\geq 1$ and $\delta, R>0$ such that
$$
\PP_{\mu}^x\left(\tau_K>b\ell, S_{b\ell}\in K_{\ell}, \max_{k\leq b\ell}\Vert S_k-x\Vert_J\leq R \right)\geq \gamma^\ell,
$$
for all $\ell\geq 1$ and $x\in K_\delta$.
\end{proposition}
Now, as soon as the random walk has reached a point $y$ at a distance $\geq R$ from the boundary, it suffices that the walk stays in $B(y,R)$ so as to be sure that it will not leave the cone.
This simple observation will enable us to derive the theorem from the second proposition:
\begin{proposition}
\label{prop:path_staying_in_a_ball}
Assume $(\widetilde{S}_n)$ is square integrable random walk with mean $m=0$ and any variance-covariance matrix. Then
$$
\lim_{R\to\infty}\liminf_{n\to\infty}\PP^0_{\mu}\left(\max_{k\leq n}\Vert \widetilde{S}_k\Vert\leq R\right)^{1/n}=1.
$$
\end{proposition}

The proofs of Proposition~\ref{prop:n_steps_push_forward} and~\ref{prop:path_staying_in_a_ball} are deferred to section~\ref{subsec:push_walk_deep}
and section~\ref{subsec:exit_time_ball}, respectively.
First of all, let us explain precisely how the combination of those two propositions leads to Theorem~\ref{prop:centered_case_extended}.

\subsection{Proof of Theorem~\ref{prop:centered_case_extended}}

Proposition~\ref{prop:n_steps_push_forward} ensures the existence of $\gamma>0$, $b\geq 1$ and $\delta, R_0>0$ such that
$$
\PP_{\mu}^x\left(\tau_K>b\ell, S_{b\ell}\in K_{\ell}, \max_{k\leq b\ell}\Vert S_k-x\Vert_J\leq R_0 \right)\geq \gamma^\ell,
$$
for all $\ell\geq 1$ and $x\in K_\delta$.

Let $\epsilon>0$ be given. Applying Proposition~\ref{prop:path_staying_in_a_ball} to the centered random walk
$\widetilde{S}_n=S_n-nm$, we obtain the existence of a number $R\geq R_0$ such that
$$
\PP^0_{\mu}\left(\max_{k\leq n}\Vert \widetilde{S}_k\Vert\leq R-R_0\right)\geq (1-\epsilon)^n,
$$
for all $n$ large enough. Now fix $\ell\geq R-R_0$ and suppose that $y\in K_{\ell}\cap\close{B_J(x,R_0)}$.
If 
$$\max_{k\leq n}\Vert \widetilde{S}_k-y\Vert\leq R-R_0(\leq\ell),$$
then:
\begin{enumerate}
\item Clearly, $\widetilde{S}_k$ belongs to $K$ for all $k\leq n$. Since $S_k=\widetilde{S}_k+km$ and $m$ belongs to $K$, the same is true for $S_k$, thus $\tau_K>n$.
\item For all $k\leq n$, we have
$$ \Vert S_k-x\Vert_J=\Vert \widetilde{S}_k-x\Vert_J\leq \Vert \widetilde{S}_k-y\Vert+\Vert y-x\Vert_J\leq R.$$
\end{enumerate}
Therefore, if we consider only trajectories such that $S_{b\ell}\in K_\ell$ and $\Vert S_{b\ell}-x\Vert_J\leq R_0$, and then use the Markov property at time $b\ell$, we obtain the lower bound
\begin{equation*}
\begin{split}
\PP_{\mu}^x\left(\tau_K>n, \max_{k\leq n}\Vert S_k-x\Vert_J \leq R\right)
&\geq \gamma^\ell \times \inf_{y}\PP_{\mu}^y\left(\tau_K>n-b\ell, \max_{k\leq n-b\ell}\Vert S_k-x\Vert_J\leq R\right)\\
&\geq \gamma^\ell \times \inf_{y}\PP_{\mu}^y\left(\max_{k\leq n-b\ell}\Vert \widetilde{S}_k-y\Vert\leq R-R_0\right)\\
&\geq \gamma^\ell \times (1-\epsilon)^{n-b\ell},
\end{split}
\end{equation*}
where the infimum is taken over all $y\in K_{\ell}\cap\close{B_J(x,R_0)}$.
Consequently
$$
\liminf_{n\to\infty}\PP_{\mu}^x\left(\tau_K>n, \max_{k\leq n}\Vert S_k-x\Vert_J \leq R\right)^{1/n}\geq 1-\epsilon,
$$
and the theorem is proved.

\subsection{Pushing the walk deep inside the cone}
\label{subsec:push_walk_deep}

This section is devoted to the proof of Proposition~\ref{prop:n_steps_push_forward}. 
In what follows,
the distribution $\mu$ of the random walk increments is assumed to satisfy assumption (H'). 
Let $m$ be the expectation of $\mu$ and
$$F=(\ker\Gamma)^{\perp},$$
where $\Gamma$ is the variance-covariance matrix of $\mu$. It is well-known that the smallest affine subspace of $\RR^d$ with full $\mu$-probability is
$$m+F.$$
Therefore, assumption (H') ensures that there exists no $u\in Q\times\{0\}^q\setminus\{0\}$ such that $m+F\subset u^{-}$.

We define the {\em smoothed support} $\G$ of the random walk as
$$\G=\RR_+m+F,$$
and notice that, if started at any point in $\G$, the random walk stays in $\G$ forever.

In addition, we assume that 
$$m\in Q\times\{0\}^q,$$ and define the set
$$J=\{j\in\inter{1}{d} : m_j=0\}\supset\inter{p+1}{q}.$$
Finally, let 
$$B_J(0,R)=\{x\in\RR^d : \Vert x\Vert_J <R\},$$
where $\Vert * \Vert_J$ is defined as in \eqref{eq:norme_ji}.

\subsubsection{Some geometry}
We collect here two technical lemmas related to the geometry of the problem.
The first one asserts that the affine support of $\mu$ meets the interior of the cone $K=Q\times \RR^q$. This is crucial since otherwise we couldn't expect the walk to go deep inside the cone.
\begin{lemma}
\label{lem:support_vs_cone}
Assume (H') is satisfied and $m\in Q\times\{0\}^q$. Then
\begin{equation}
\label{eq:basic_geometry}
(m+F)\cap\open{K}\not=\emptyset
\end{equation}
Moreover, for all $x\in K\cap\G$, 
\begin{equation}
\label{eq:basic_geometry_extended}
(x+m+F)\cap\open{K}\cap B_J(0,1)\not=\emptyset.
\end{equation}

\end{lemma}
\begin{proof} 
We assume $(m+F)\cap\open{K}=\emptyset$, and infer the existence of some $u\in Q\times\{0\}^q$ such that $(m+F)\subset u^{\perp}$, thus contradicting the assumption (H').

Let us first consider the case where $F=H$ is a hyperplane. Then there exists $u\not=0$ such that $H=u^\perp$.
We shall prove that $m\in u^{\perp}$. Suppose on the contrary that $\sclr{m}{u}\not=0$, then,
possibly changing $u$ to $-u$, we can assume that $\sclr{m}{u}> 0$. 
Now, using the homogeneity of $H$ and $\open{K}$, we see that $\open{K}$ does not intersect with
$$\bigcup_{\lambda>0}(\lambda m+H)=\bigcup_{\lambda>0}(\lambda u+u^{\perp})=\{x : \sclr{x}{u}>0\}.$$
Therefore, $\open{K}$ is included in $u^{-}$. But since $m\in K=\overline{(\open{K})}$ (this equality holds for any convex set with non-empty interior), we obtain that
$\sclr{m}{u}\leq 0$, which contradicts our hypothesis. Hence, $m$ belongs to $u^{\perp}$ and $m+H=u^{\perp}$. Finally, the non-intersecting hypothesis rewrites
$u^{\perp}\cap \open{K}=\emptyset$ and is easily seen to be equivalent to $u\in \pm Q\times\{0\}^q$.

We now turn to the general case where $F$ is any linear subspace. Since $m+F$ and $\open{K}$ are two disjoints convex sets, it follows from the Hyperplane separation theorem that there exists an affine hyperplane $H_m$ that separates $m+F$ and $\open{K}$. But, since $m$ belongs to both $m+F$ and $K=\overline{\open{K}}$, it must belong to $H_m$, and therefore $H_m=m+H$, where $H$ is a linear hyperplane. Now, $F$ being a linear subspace, it can't be on one side of $H$ unless it is contained in $H$.
Therefore, we obtain that $m+F\subset m+H$ and $(m+H)\cap \open{K}=\emptyset$, and equation \eqref{eq:basic_geometry} follows by applying the first part of the proof to $m+H$. 

Let us now show that \eqref{eq:basic_geometry} implies \eqref{eq:basic_geometry_extended}. Since $(m+F)\cap \open{K}$ is non-empty, there is some $f_0\in F$ such that $m+f_0\in\open{K}$. Therefore 
$$m+\alpha f_0\in\open{K}$$
for all $\alpha\in (0,1]$ (since $m\in K$ and $K$ is convex). Fix such an $\alpha$ so small that $\Vert \alpha f_0\Vert_J<1$.
For $x\in K\cap\G$, write $x=\lambda_1m+f_1$ with $\lambda_1\geq 0$ and $f_1$ in $F$, and set $f=\alpha f_0-f_1$. 
Then,
$$x+m+f=\lambda_1m+(m+\alpha f_0)\in \open{K}.$$
(since $\lambda_1m\in K$, $m+\alpha f_0\in\open{K}$ and $K+\open{K}\subset\open{K}$.)
In addition,
$$\Vert x+m+f\Vert_J=\Vert \alpha f_0\Vert_J<1,$$
thus proving that $x+m+f\in \open{K}\cap B_J(0,1)$.
\end{proof}

The second lemma is a technical tool. 
For any $x,y\in\RR^d$, we write $y\leq x$ iff $x-y\in\RR_+$.

\begin{lemma}
\label{lem:minimizing_sequence}
Let $(x_n)$ be a sequence in $K\cap\close{B_J(0,1)}\cap\G$. There exists a bounded sequence $(y_n)$ in $K\cap\close{B_J(0,1)}\cap\G$, such that
$$y_n\leq x_n\quad\mbox{ and }\quad y^{(J)}_n=x^{(J)}_n\quad \mbox{ for all }n.$$
\end{lemma}
\begin{proof}
Since $\G=\RR_+m+F$, where $F$ is a linear subspace of $\RR^d$, there exist $L\in M_d(\RR)$ and a linear form $\phi:\RR^d\to\RR$, such that
$$\G=\{x\in\RR^d : L(x)=0, \phi(x)\geq 0\}.$$
Recall that
$$\RR^{(J)}=\{x\in\RR^d : x_i=0,\forall i\notin J\}.$$
Let $\RR^{(I)}$ be its orthogonal complement, i.e.
$$\RR^{(I)}=\{x\in\RR^d : x_i=0,\forall i\in J\},$$
and write $x=x^{(I)}+x^{(J)}$ the orthogonal decomposition with respect to $\RR^{(I)}\oplus\RR^{(J)}$.
Let $x_n=x_n^{(I)}+x_n^{(J)}$ be an element of $K\cap\close{B_J(0,1)}\cap\G$ and define
$$P=\{z\in \RR_{+}^{(I)} : L(z)=-L(x_n^{(J)}), \phi(z)\geq -\phi(x_n^{(J)})\}.$$
For all $z\in \RR_{+}^{(I)}$, notice that $z\in P$ iff $z+x_n^{(J)}\in \G$. Therefore, $x_n^{(I)}\in P$.
It follows from Corollary~\ref{cor:polyhedron_2} that there exists $y_n^{(I)}\in P$ such that $y_n^{(I)}\leq x_n^{(I)}$ and
$$\Vert y_n^{(I)}\Vert\leq M(\Vert L(x_n^{(J)})\Vert+\vert\phi(x_n^{(J)})\vert),$$
where $M=M(L,\phi)$ only depends on $L$ and $\phi$ (and not on $x_n$). Setting $y_n=y_n^{(I)}+x_n^{(J)}$ thus gives a bounded sequence in 
$K\cap\close{B_J(0,1)}\cap\G$ that satisfies the conditions of the lemma.
\end{proof}

\subsubsection{Proof of Proposition \ref{prop:n_steps_push_forward}}

We begin with a lemma that asserts the existence of a time $b$ and a radius $R_0$ such that the random walk started at $x\in K\cap\G$ with $\Vert x\Vert_J\leq R_0$ will be  at a distance $\geq 1$ from the boundary of $K$ at time $b$ and still located in $\close{B_J(0,R_0)}$ with a probability that is bounded from below by some positive constant, uniformly in $x$. 

\begin{lemma}
\label{lem:one_step_push_forward}
There exist $b\geq 1$ and $R_0>0$ such that
$$
\inf_{x\in K\cap\close{B_J(0,R_0)}\cap\G}\PP_{\mu}^x\left(S_{b}\in K_{1}, \Vert S_b\Vert_J\leq R_0 \right)>0.
$$
\end{lemma}
\begin{proof} Clearly, the lemma will follow from the existence  of an integer $n\geq 1$ such that
$$
\inf_{x\in K\cap\close{B_J(0,1)}\cap\G}\PP_{\mu}^{x\sqrt{n}}\left(S_{n}\in K_{1}, \Vert S_n\Vert_J\leq \sqrt{n}\right)>0.
$$
So, let us assume that this assertion is false. Then, we can find a sequence of points $x_n\in K\cap\close{B_J(0,1)}\cap\G$ such that
$$
p_n:=\PP_{\mu}^{x_n\sqrt{n}}\left(S_{n}\in K_{1}, \Vert S_n\Vert_J\leq \sqrt{n}\right)\to 0.
$$
Thanks to Lemma~\ref{lem:minimizing_sequence}, we can assume that $(x_n)$ is bounded, because for any sequence $(y_n)$
with the same properties as in this lemma, the probability $p_n$ where $x_n$ is replaced by $y_n$ is smaller than $p_n$, by inclusion of events. Furthermore, by extracting a subsequence, it can be assumed without loss of generality that $(x_n)$ converges to some element $x$ of the closed set $K\cap\close{B_J(0,1)}\cap\G$.

Now, let $\widetilde{S}_n=S_n-nm$ denote the centered random walk associated with $S_n$. Since $\Vert *\Vert_J$ is left invariant by a translation by $m$, the probability $p_n$ can be written as
$$
p_n=\PP_{\mu}^{0}\left(x_n\sqrt{n}+\widetilde{S}_{n}\in K_{1}-nm, \Vert x_n\sqrt{n}+\widetilde{S}_n\Vert_J\leq \sqrt{n}\right).
$$
Let $\eta>0$ be fixed. For all $n\geq 1/\eta$, holds the inclusion 
$$(K_{1}-nm)/\sqrt{n}\subset K_{\eta}-m.$$
Therefore, the probability $p_n$ is bounded from below by
$$\PP_{\mu}^{0}\left(x_n+\widetilde{S}_{n}/\sqrt{n}\in K_{\eta}-m, \Vert x_n+\widetilde{S}_n/\sqrt{n}\Vert_J\leq 1\right).$$
Since $x_n+\widetilde{S}_{n}/\sqrt{n}$ converges in distribution to $x+X$, where $X$ denotes a random variable with $\N(0,\Gamma)$ Gaussian distribution, we can use the Portmanteau theorem and let then $\eta\downarrow 0$ to get the lower bound
$$\liminf_{n\to\infty}p_n\geq \PP\left(x+X\in (\open{K}-m)\cap B_J(0,1)\right).$$
The random variable $x+X$ admits a positive density with respect to Lebesgue measure on the affine space $x+F$ where $F=(\ker\Gamma)^\perp$,
and
$$O:=(x+F)\cap(\open{K}-m)\cap B_J(0,1)$$
is an open subset of $x+F$. Thus it suffices to prove that $O$ is non-empty to obtain a contradiction. But this is precisely what asserts Lemma~\ref{lem:support_vs_cone}:
Indeed, since $x$ belongs to $K\cap \G$, we obtain that
$$(x+m+F)\cap\open{K}\cap B_J(0,1)\not=\emptyset.$$
Hence, substracting $m$ on both sides (recall that $\Vert *\Vert_J$ is left invariant by a translation by $m$) gives
$$O\not=\emptyset.$$
This implies
$$\PP(x+X\in O)>0,$$
thus contradicting our assumption that $\liminf p_n=0$. Therefore, the lemma is proven.
\end{proof}

By Lemma~\ref{lem:one_step_push_forward}, there exist $b\geq 1$, $R_0>0$ and $\gamma>0$, such that
\begin{equation}
\label{eq:push_and_back}
\PP_{\mu}^x\left(S_{b}\in K_{1}, \Vert S_b\Vert_J\leq R_0 \right)\geq 2\gamma,
\end{equation}
for all $x\in K\cap\close{B_J(0,R_0)}\cap\G$. Let us choose $\delta>0$ and $R\geq R_0$ such that
$$
\PP_{\mu}^0\left(\tau_{K_{-\delta}}>b,  \max_{k\leq b}\Vert S_k\Vert_J\leq R-R_0 \right)\geq 1-\gamma.
$$
Then, by inclusion of events, we also have
\begin{equation}
\label{eq:security}
\PP_{\mu}^x\left(\tau_{K_{-\delta}}>b,  \max_{k\leq b}\Vert S_k\Vert_J\leq R \right)\geq 1-\gamma,
\end{equation}
for all $x\in K\cap\close{B_J(0,R_0)}\cap\G$. Indeed, this follows from the relation $K+K_{-\delta}\subset K_{-\delta}$ and the triangle inequality for $\Vert *\Vert_J$.
Now, combining \eqref{eq:push_and_back} and \eqref{eq:security}, we obtain that
\begin{equation}
\label{eq:push_and_back_with_security}
\PP_{\mu}^x\left(\tau_{K_{-\delta}}>b, S_{b}\in K_{1}, \max_{k\leq b}\Vert S_k\Vert_J\leq R, \Vert S_b\Vert_J\leq R_0\right)\geq \gamma,
\end{equation}
for all $x\in K\cap\close{B_J(0,R_0)}\cap\G$.
Set
$$
p_\ell(x)=\PP_{\mu}^x\left(\tau_{K_{-\delta}}>b\ell, S_{b\ell}\in K_{\ell}, \max_{k\leq b\ell}\Vert S_k\Vert_J\leq R \right).
$$
Notice that $K_1+K_\ell\subset K_{\ell+1}$. Hence, 
if we consider only trajectories such that $S_b\in K_1$ and $\Vert S_b\Vert_J\leq R_0$, and then use the Markov property at time $b$,
we get the lower bound 
$$
p_{\ell+1}(x)\geq \gamma \times\inf_{y\in K\cap\close{B_J(0,R_0)}\cap\G}p_\ell(y),
$$
for all $x\in K\cap\close{B_J(0,R_0)}\cap\G$. This proves that
$p_\ell(0)\geq \gamma^\ell$
for all $\ell\geq 1$, and Proposition~\ref{prop:n_steps_push_forward} follows by inclusion of events since $K_\delta+K_{-\delta}\subset K$ and $K_\delta+K_\ell\subset K_\ell$.

\subsection{On the exit time from a ball}
\label{subsec:exit_time_ball}
This section is devoted to the proof of Proposition~\ref{prop:path_staying_in_a_ball}.
In what follows, the abbreviation ``f.s.'' stands for ``for some''.

\subsubsection{Preliminary estimate for Brownian motion}

Let $(B_t)$ denotes a true $d$-dimensional Brownian motion, i.e. the image of a standard $d$-dimensional Brownian motion (meaning a collection of $d$ independent one-dimensional Brownian motions) by an invertible linear transformation.

\begin{lemma} 
\label{lem:brownian_motion_estimate}
For every $\epsilon>0$, there exist $\delta>0$ and $0<\alpha<1$ such that
$$\PP^x\left(\Vert B_t\Vert< 1-\delta\mbox{ f.s. }t\in[\alpha,1]\right)\geq 1-\epsilon$$
for all $x\in\close{B(0,1)}$.
\end{lemma}

\begin{proof} 
Suppose on the contrary that there is some $\epsilon_0>0$ for which we can pick $\alpha_n\downarrow 0$, $\delta_n\downarrow 0$ and $x_n\in \close{B(0,1)}$ such that
$$
p_n:=\PP^{x_n}\left(\Vert B_t\Vert< 1-\delta_n\mbox{ f.s. }t\in[\alpha_n,1]\right)< 1-\epsilon_0
$$
for all $n$. By compactness, it can also be assumed that $\Vert x_n-x\Vert\to 0$ for some $x\in\close{B(0,1)}$.
Now, for any $\eta>0$, we have
$$
p_n\geq\PP^x\left(\Vert B_t\Vert< 1-\eta\mbox{ f.s. }t\in[\alpha_n,1]\right)
$$
as soon as $\Vert x_n-x\Vert+\delta_n\leq \eta$. 
Hence, taking the limit on both sides and letting $\eta\downarrow 0$ gives
$$\liminf_{n\to\infty} p_n\geq\PP^x\left(\Vert B_t\Vert< 1\mbox{ f.s. }t\in(0,1]\right).$$

But it follows from the classical cone condition (as found in \cite[Proposition 3.3]{PorSto78} for example) applied to the ball $B(0,1)$ that
$x$ is regular for $B(0,1)$, i.e.~$B_t$ immediately visits $B(0,1)$ with full probability.
Therefore, the last inequality reads
$$\liminf_{n\to\infty} p_n\geq 1$$
and contradicts our assumption.

\end{proof}

\subsubsection{Application to random walks}

In this subsection, $(S_n)\in \RR^d$ is a square integrable random walk with increments distribution $\mu$, mean $m=0$ and any covariance matrix $\Gamma$.

The proof of Proposition \ref{prop:path_staying_in_a_ball} is based on the following basic idea. Given $\epsilon>0$, find $R>0$ and a time $n_0\geq 1$ such that the random walk started at $0$ returns to $0$ at time $n_0$ without leaving the ball $B(0,R)$ with probability $\geq 1-\epsilon$. If this can be done, then the result follows by concatenation (i.e~Markov property). But this is asking for a property stronger than recurrence, thus we can not hope for such a simple argument. Instead of a return to $0$, we can ask for a return in some ball $B(0,R_0)$, with $R_0\leq R$. But then, in view of using concatenation, we need the previous probability to be greater than $1-\epsilon$ uniformly for any starting point in the same ball $B(0,R_0)$.
Lemma~\ref{lem:pseudo-recurrent_path_with_size_control} below provides a result in this spirit that is sufficient for our purpose.

\begin{lemma}
\label{lem:pseudo-recurrent_path_with_size_control} 
Suppose (here only) that the covariance matrix $\Gamma$ is non-degenerate.
Then, for every $\epsilon>0$, there exist $0<R_0<R$ and $1\leq \ell_0\leq n_0$ such that
$$
\PP^{x}_{\mu}\left(\max_{k\leq n_0} \Vert S_n\Vert\leq R \mbox{ and } \Vert S_k\Vert \leq R_0 \mbox{ f.s. } k\in\inter{\ell_0}{n_0}\right)\geq 1-\epsilon
$$
for all $x\in \close{B(0,R_0)}$.
\end{lemma}
\begin{proof}
Let $\epsilon>0$ be given and fix $\delta>0$ and $0<\alpha<1$ so that the conclusion of Lemma \ref{lem:brownian_motion_estimate} holds.
Also, fix a parameter $\beta\in(0,\alpha)$.

Let $(x_n)$ be a sequence in $\close{B(0,1)}$ that converges to some $x\in\close{B(0,1)}$, and set
$$
p_n=\PP_{\mu}^{x_n\sqrt{n}}(\Vert S_k\Vert\leq\sqrt{n}\mbox{ f.s. } k\in [\beta n, n])
$$
We shall prove that
\begin{equation}
\liminf_{n\to\infty}p_n\geq 1-\epsilon.
\end{equation}
To do this, we consider the process with continuous path $(Z_n(t), t\in[0,1])$ defined by
$$
Z_n(t)=\frac{S_{[nt]}}{\sqrt{n}}+(nt-[nt])\frac{\xi_{[nt]+1}}{\sqrt{n}},
$$
where $\xi_k=S_k-S_{k-1}$ and $[a]$ denotes the integer part of $a$.  The probability $p_n$ can then be written in terms of the process $Z_n$ as $\PP_{\mu}^0(A_n)$, where
$$
A_n=\left\{\Vert x_n+Z_n(t)\Vert\leq 1\mbox{ f.s. } t=k/n\in[\beta,1]\right\}.
$$
We wish to use the Functional Central Limit Theorem \cite[Theorem 10.1]{Bil68} together with the Portmanteau theorem \cite[Theorem 2.1]{Bil68} in order to obtain a lower bound for the probability of this event, but
the condition that $t$ be rationnal ($t=k/n$) can not be handled directly, and must therefore be relaxed. To this end, we define
$$
\widetilde{A_n}=\left\{\Vert x+Z_n(t)\Vert< 1-\delta\mbox{ f.s. } t\in[\alpha,1]\right\}
$$
and
$$
B_n=\left\{\max_{k=0\ldots n}\Vert \xi_k\Vert > (\delta/2)\sqrt{n}\right\}.
$$
Since $\Vert Z_n(t)- Z_n(k/n)\Vert \leq \Vert \xi_{k+1}\Vert/\sqrt{n}$ for $k/n\leq t<(k+1)/n$, we have the inclusion of events 
$$
\widetilde{A_n}\cap \close{B_n}\subset A_n
$$
as soon as $\Vert x_n-x\Vert\leq \delta/2$ and $\alpha-\beta>1/n$.
Thus, for all sufficiently large $n$, we have
$$
p_n\geq \PP^{0}_{\mu}(\widetilde{A_n})-\PP^0_{\mu}(B_n).
$$
Furthermore, it is a basic result in probability theory that $\PP(B_n)\to 0$ for any i.i.d. sequence of zero-mean square integrable random variables $(\xi_k)$.
Hence, we are left to prove that
$$
\liminf_{n\to\infty}\PP^0_{\mu}(\widetilde{A_n})\geq 1-\epsilon.
$$
Now, since $\Gamma$ is non-degenerate, the Functional Central Limit Theorem asserts that $Z_n$ converges in distribution to a true $d$-dimensional Brownian motion (with covariance matrix $\Gamma$) on the space of continuous functions $w:[0,1]\to \RR^d$ equipped with the topology of uniform convergence.
Since the set of continuous functions $w:[0,1]\to \RR^d$ such that $\Vert w(t)\Vert< 1-\delta$ for some $t\in[\alpha,1]$ is open with respect to the topology of uniform convergence, it follows from the Portmanteau theorem that
$$
\liminf_{n\to\infty} \PP^0_{\mu}(\widetilde{A_n})\geq \PP^0\left(\Vert x+B_t\Vert< 1-\delta\mbox{ f.s. }t\in[\alpha,1]\right)\geq 1-\epsilon,
$$
where the lower bound $1-\epsilon$ comes from our choice of $\delta$ and $\alpha$.
This proves our first claim that $\liminf_n p_n\geq 1-\epsilon$. By a standard compactness argument, this immediately implies that
$$
\inf_{x\in \close{B(0,1)}}\PP_{\mu}^{x\sqrt{n}}(\Vert S_k\Vert\leq\sqrt{n}\mbox{ f.s. } k\in[\beta n,n])\geq 1-2\epsilon
$$
for all sufficiently large $n$. Fix such $n_0>1/\beta$ and set $\ell_0=[\beta n_0]$ and $R_0=\sqrt{n_0}$. Then we have $0<\ell_0<n_0$, and the last inequality can be rewritten as
$$
\inf_{x\in \close{B(0,R_0)}}\PP^{x}_{\mu}\left(\Vert S_k\Vert \leq R_0 \mbox{ f.s. } k\in\inter{\ell_0}{n_0}\right)\geq 1-2\epsilon
$$
In order to complete the proof, it suffices to notice that, for all $x\in \close{B(0,1)}$ and $R\geq R_0$,  we have
$$
\PP_{\mu}^x\left(\max_{k\leq n_0}\Vert S_k\Vert\leq R\right)\geq\PP_{\mu}^0\left(\max_{k\leq n_0}\Vert S_k\Vert\leq R-R_0\right).
$$
Since the last probability goes to $1$ as $R\to\infty$, it is bounded from below by $1-\epsilon$ for all sufficiently large $R$, and we conclude that for such a choice of $R$, we have
$$\PP^{x}_{\mu}\left(\max_{k\leq n_0}\Vert S_k\Vert\leq R\mbox{ and }\Vert S_k\Vert \leq R_0 \mbox{ f.s. } k\in\inter{\ell_0}{n_0}\right)\geq 1-3\epsilon$$
for all $x\in \close{B(0,R_0)}$.
\end{proof}

\subsubsection{Proof of Proposition~\ref{prop:path_staying_in_a_ball}}
First of all, we notice that the variance-covariance matrix $\Gamma$ can be assumed to be non-degenerate. Otherwise, the random walk lives on $(\ker\Gamma)^{\perp}$ where it has a non-degenerate variance-covariance matrix.
Since the projection of a $d$-dimensional ball is still a ball in $(\ker\Gamma)^{\perp}$, the result will follow by application
of the non-degenerate case to the projected random walk. 

As explained earlier, the idea is now to concatenate the ``high-probability path'' given by Lemma~\ref{lem:pseudo-recurrent_path_with_size_control}.
Let $\epsilon>0$ be given. Thanks to Lemma \ref{lem:pseudo-recurrent_path_with_size_control}, we can find  
$\delta>0$, $0<R_0<R$ and $1\leq \ell_0\leq n_0$ such that
$$\inf_{x\in \close{B(0,R_0)}}\PP_{\mu}^x\left(\max_{k\leq n_0}\Vert S_n\Vert\leq R, H\leq n_0\right)\geq 1-\epsilon,$$
where $H$ denotes the first hitting time of the ball $\close{B(0,R_0)}$ after time $\ell_0$.

For $x\in \close{B(0,R_0)}$ and $n\geq n_0$, we use the strong Markov property at time $H$ to get
\begin{align*}
\PP^{x}_{\mu}&\left(\max_{k\leq n}\Vert S_k\Vert\leq R\right)\\
& \geq \PP^{x}_{\mu}\left(\max_{k\leq n}\Vert S_k\Vert\leq R, H\leq n_0\right)\\
     & \geq \EE_{\mu}^{x}\left(\max_{k\leq H}\Vert S_k\Vert\leq R, H\leq n_0,\PP_{\mu}^{S_H}\left(\max_{k\leq n-j}\Vert S_k\Vert\leq R\right)_{| j=H}\right)\\
     & \geq \PP^x_{\mu}\left(\max_{k\leq H}\Vert S_k\Vert\leq R, H\leq n_0\right)\times 
		\inf_{y\in \close{B(0,R_0)}}\PP_{\mu}^y\left(\max_{k\leq n-\ell_0}\Vert S_k\Vert\leq R\right).
\end{align*}
Thus  
$$
R(n):=\inf_{x\in \close{B(0,R_0)}}\PP^{x}_{\mu}\left(\max_{k\leq n}\Vert S_k\Vert\leq R\right)
$$ 
satisfies the inequality
$$
R(n)\geq (1-\epsilon)R(n-\ell_0)
$$
for all $n\geq n_0$. Since $R(n)$ is clearly decreasing, for $n_0+k\ell_0\leq n< n_0+(k+1)\ell_0$, we obtain
$$
R(n)\geq R(n_0+(k+1)\ell_0)\geq (1-\epsilon)^{k+1}R(n_0)\geq (1-\epsilon)^{n+2}.
$$
Hence,
$$
\liminf_{n\to\infty} R(n)^{1/n}\geq 1-\epsilon
$$
and this proves the proposition.

\appendix
\section{Minimal points of a polyhedron}

This independent and (nearly) self-contained section provides the material we need for the proof of Lemma~\ref{lem:minimizing_sequence}. The notion of minimality introduced here and the related results are certainly not new, but we were not able to find any reference for them.
The arguments developed here are highly inspired by the standard ideas of linear programming theory, as can be found in the book
\cite{Cia82} for example.

Let $1\leq m\leq n$. We consider here the polyhedron
$$P=\{x\in\RR_+^n : \sum_{i=1}^n x_iC_i=b\},$$
where $C_i, b$ are any vectors of $\RR^m$, and $x_i$ denote the coordinates of $x$ in the standard basis.

If $P$ is not empty (a condition that we assume from now on), it is a closed convex set. Its extremal points are called the {\em vertices} of $P$. Note that $0$ is a vertex iff $b=0$. There exist a simple and well known characterization of vertices. To $x\in P$, let us associate the subset $I(x)$ of indices $i$ such that $x_i>0$. Then $x\not=0$ is a vertex of $P$ iff the vectors $\{C_i, i\in I(x)\}$ are linearly independent (see \cite[Th\'eor\`eme 10.3-1]{Cia82}).
This proves that there is only a finite number of vertices. The fact that there always exist (at least) a vertex is proved in
\cite[Th\'eor\`eme 10.3-3]{Cia82} for example, and will follow as a by-product of our analysis.

For $x,y\in \RR^n$, we write $y\leq x$ iff $x-y\in\RR_+^n$ and $y<x$ iff $y\leq x$ and $y\not=x$. Hence $x>0$ means that $x\geq 0$ and one of its coordinates (at least) is $>0$.
We shall say that $x\in P$ is {\em minimal} if $P$ does not contain any $y<x$.

Let us begin with a useful characterization of minimality:
\begin{lemma} 
An element $x\not=0$ of $P$ is minimal iff the vectors $\{C_i, i\in I(x)\}$ are positively independent, that is:
$$\sum_{i\in I(x)}u_iC_i=0\mbox{ and } u_i\geq 0 \mbox{ for all }i\in I(x)\Rightarrow u_i=0\mbox{ for all }i\in I(x).$$
\end{lemma}
\begin{proof} 
Let $x$ be an element of $P$. Assume first that $x$ is not minimal. Then there exist $y\in P$ such that $y<x$.
If $i\not\in I(x)$, then $y_i=x_i=0$ (since $0\leq y_i\leq x_i=0$). Thus, there is at least one $i\in I(x)$ such that $y_i<x_i$.

Set $u_i=x_i-y_i$ for $i\in I(x)$. Then, the $u_i$'s are $\geq 0$ and at least one of them is $>0$.
Furthermore, 
$$\sum_{i\in I(x)}u_iC_i=\sum_{i\in I(x)}x_iC_i-\sum_{i\in I(x)}y_iC_i=\sum_{i=1}^nx_iC_i-\sum_{i=1}^ny_iC_i=b-b=0.$$
Hence, the vectors $\{C_i, i\in I(x)\}$ are not positively independent.

Conversely, suppose that there exist non-negative numbers $u_i,i\in I(x)$, not all zero, such that
$$\sum_{i\in I(x)}u_i C_i=0.$$
For $i\not\in I(x)$, set $u_i=0$, and let $u\in \RR_+^n$ denote the vector with coordinates $u_i$. Since one of the $u_i$'s is $>0$, we have $y=x-tu<x$ for any $t>0$. But, if $t>0$ is small enough, then $y$ also belongs to $\RR_+^n$ (this follows from the fact that $x_i=0\Rightarrow u_i=0$). Furthermore,
$$\sum_{i=1}^ny_iC_i=\sum_{i\in I(x)}(x-tu)=\sum_{i\in I(x)}x_iC_i+t\sum_{i\in I(x)}u_iC_i=b.$$
Hence, we find some $y\in P$ such that $y<x$, i.e. $x$ is not minimal.
\end{proof}
From this lemma, it should be clear that any vertex of $P$ is minimal. Let us denote by $P_m$ the set of minimal points of $P$.
The main result of this section is the following:

\begin{proposition}
\label{prop:polyhedron_1}
Assume $P\not=\emptyset$. Then,
\begin{enumerate}
\item For any $x\in P$, there exist $y\in P_m$ such that $y\leq x$;
\item Every $y\in P_m$ is a convex combination of the vertices of $P$;
\item There exists a number $M$, only depending on the $C_i$'s, such that
$P_m$ is bounded by $M\Vert b\Vert$.
\end{enumerate}
\end{proposition}
\begin{proof} Let $x\in P$ and suppose that $x$ is not minimal (otherwise there is nothing to prove).
Then, by definition, there exist real numbers $u_i\geq 0, i\in I(x)$, such that
$$\sum_{i\in I(x)}u_iC_i=0$$
and at least one of the $u_i$'s is $>0$. For $i\not\in I(x)$, set $u_i=0$, and let $u\in \RR_+^n$ denote the vector with coordinates $u_i$. Now, we let $t\geq 0$ increase until the first time when $x-tu$ has a new $0$ coordinate, that is we define
$$t_0=\min\{x_i/u_i: u_i>0\}>0$$
and set
$$y=x-t_0u.$$
Clearly, this new point $y$ satisfies the inequality $0\leq y<x$, and
$$\sum_{i=1}^n y_i C_i=\sum_{i\in I(x)}(x_i-t_0u_i)C_i=b.$$
Thus $y<x$ also belongs to $P$ and, furthermore, satisfies the strict inclusion relation $I(y)\subsetneq I(x)$.
If $y$ is not minimal, we repeat the argument (with $y$, and so on) until a minimal point is reached. The process indeed terminates since the set of positive indices $I(*)$ is finite and strictly decreasing along the process. This proves the first part of the proposition.

The proof of the second part is quite similar. Let $x\in P$ be minimal and suppose it is not a vertex of $P$.
Then, the vectors $\{C_i, i\in I(x)\}$ are positively independent (by minimality) but not independent. Therefore, there exist real numbers $\{u_i, i\in I(x)\}$, not all zero, such that 
$$\sum_{i\in I(x)}u_iC_i=0$$
and at least one of them is negative ($<0$) and another is positive ($>0$) (since else $-u$ would contradict the assumption).
For $i\not\in I(x)$, set $u_i=0$, and let $u\in \RR^n$ denote the vector with coordinates $u_i$. Now, we let $t\geq 0$ increase until the first time when $x+tu$ has a new $0$ coordinate, that is we define 
$$t^+=\min\{-x_i/u_i: u_i<0\}>0\quad\mbox{ and }\quad x^+=x+t^+u.$$
Similarly, we define
$$t^-=\min\{x_i/u_i: u_i>0\}>0\quad\mbox{ and }\quad x^-=x-t^-u.$$
These new points $x^+$ and $x^-$ belong to $P_m$ since they clearly are $\geq 0$, satisfy the equality
$$\sum_{i=1}^n x_i^\pm C_i=\sum_{i\in I(x)}(x_i\pm t^\pm u_i)C_i=b,$$
and the strict inclusion relation $I(x^\pm)\subsetneq I(x)$. 
Furthermore, $x$ is a convex combination of $x^+$ and $x^-$ since
$$(t^++t^-)x=t^+ x^-+t^- x^+.$$
If $x^+$ and $x^-$ are both vertices of $P$, then the proof is finished. Else, we repeat the argument until $x$ is written as a convex combination of vertices. The process indeed terminates since the sets of positive indices $I(*)$ are finite and strictly decreasing along the process. This proves the second part of the proposition.

For the proof of the third assertion, we first notice that the Manhattan norm $\Vert x\Vert_1=\sum_{i=1}^n\vert x_i\vert$ coincides with the linear form $\sum_{i=1}^n x_i$ on $\RR_+^n$. Hence, it easily follows from the second assertion of the proposition that
$$\sup_{x\in P_m}\Vert x\Vert_1 = \max\{\Vert v\Vert_1 : v\mbox{ vertex of }P\}.$$
(Recall that there is at least one vertex and only a finite number of vertices.)
Let us denote by $\I$ the family of subsets $I\subset\inter{1}{n}$ such that $\{C_i, i\in I\}$ is a family of linearly independent vectors.
Set $$\RR^{(I)}=\{x\in\RR^n : x_i=0, \forall i\not\in I\}.$$
For all $I\in \I$, the linear mapping
$$C_I:\RR^{(I)}\to \RR^m, u\mapsto \sum_{i\in I}u_i C_i$$
is a bijection onto its image. Let us denote $C_I^{-1}$ its inverse.

If $v$ is a vertex of $P$, then $I=I(v)$ belongs to $\I$, and
$$\sum_{i\in I}v_iC_i=b \Leftrightarrow (v_i)_{i\in I}=C_I^{-1}b$$
Therefore, 
$$\Vert v\Vert_1=\Vert (v_i)_{i\in I}\Vert_1\leq \Vert C_I^{-1}\Vert_1\Vert b\Vert_1,$$
and $P_m$ is thus bounded by $M\Vert b\Vert_1$, where
$M=\max_{I\in \I}\Vert C_I^{-1}\Vert_1$.
\end{proof}

In view of application to the proof of Lemma~\ref{lem:minimizing_sequence}, we need to extend some of the consequences of Proposition~\ref{prop:polyhedron_1} to a larger class of polyhedra. So, let $L:\RR^n\to\RR^n$ be a linear mapping and 
$\phi:\RR^n\to\RR$ be a linear form. Fix $b\in\RR^n$, $c\in\RR$, and consider the polyhedron
$$P=\{x\in\RR_+^n : L(x)=b, \phi(x)\geq c\}.$$

\begin{corollary}
\label{cor:polyhedron_2}
Assume $P\not=\emptyset$. There exists a number $M$, only depending on $L$ and $\phi$, such that, 
for all $x\in P$, there exists $y\in P$ with $y\leq x$ and $\Vert y\Vert\leq M(\Vert b\Vert+\vert c\vert)$.
\end{corollary}
\begin{proof}
We use the very standard trick in linear programming problems that consists in increasing the dimension 
so that the new form of polyhedra fits with the previous one. Indeed, if we add the ``ghost'' equation $x_{n+1}=\phi(x)-c$ to the system of equalities and inequality that defines $P$, then the inequality $\phi(x)\geq c$ becomes $x_{n+1}\geq 0$ and the system of equations
$$
\begin{cases}
L(x)=b\\
x_{n+1}=\phi(x)-c
\end{cases}
$$
can be written as
$$L'(x')=b'$$
where $x'=(x,x_{n+1})$, $b'=(b,-c)$ and $L'\in M_{n+1}(\RR)$ only depends on $L$ and $\phi$.
So, define 
$$P'=\{x'\in\RR_+^{n+1}: L'(x')=b'\}.$$
Then, the mapping $\Psi: x\mapsto x'=(x,\phi(x)-c)$ is a bijection from $P$ onto $P'$.
Applying Proposition~\ref{prop:polyhedron_1} to $P'$, we obtain the existence of a number $M$, only depending on $L'$ (and thus only depending on $L$ and $\phi$), such that any minimal point $y'$ of $P'$ satisfies
$$\Vert y'\Vert_1\leq M\Vert b'\Vert_1=M(\Vert b\Vert_1+\vert c\vert).$$
We also know, from the same proposition, that for any $x'=\Psi(x)\in P'$, there exists $y'\in P'$ such that 
$y'\leq x'$ and $y'$ is minimal. Taking $y=\Psi^{-1}(y')$ gives the expected result.
\end{proof}

\section*{Acknowledgments}
The author would like to thank Kilian Raschel for his valuable comments.

\end{document}